\documentclass[11pt,reqno]{amsart}

\usepackage[english]{babel}
\usepackage{amsthm}
\usepackage{amsmath}
\usepackage{amssymb}
\usepackage{amstext}
\usepackage{latexsym}
\usepackage{enumitem}
\usepackage{mathrsfs}
\usepackage{stmaryrd}
\usepackage{comment}
\usepackage{hyperref}
\usepackage{url}
\usepackage[left=1.9cm, right=1.9cm,marginparwidth=3cm, marginparsep=0.35cm]{geometry}

\usepackage{tikz} 
\usetikzlibrary{patterns}
\usepackage{tikz-3dplot}

\theoremstyle{plain}
\newtheorem{theorem}{Theorem}
\newtheorem{lemma}[theorem]{Lemma}

\newtheorem{corollary}[theorem]{Corollary}
\newtheorem{conjecture}{Conjecture}

\providecommand{\R}{\mathbb{R}}

\providecommand{\Bcal}{\mathcal{B}}

\providecommand{\DD}{\mathsf{D}}

\providecommand{\NNb}{\mathbb{N}}

\providecommand{\RR}{\mathbb{R}}
\providecommand{\ZZ}{\mathbb{Z}}

\providecommand{\Conv}{\mathrm{Conv}}
\providecommand{\vol}{\mathrm{vol}}
\providecommand\llb{\llbracket}
\providecommand\rrb{\rrbracket}

\begin{document}
\title{The Davenport constant of balls and boxes}

\author{Benjamin Girard}
\address{Sorbonne Université, Université Paris Cité, CNRS, IMJ-PRG, \'Equipe Combinatoire et Optimisation, F-75005 Paris, France}
\email{benjamin.girard@imj-prg.fr}

\author{Alain Plagne}
\address{Centre de math\'ematiques Laurent Schwartz, CNRS, \'Ecole polytechnique, Institut polytechnique de Paris, 91128 Palaiseau cedex, France}
\email{alain.plagne@polytechnique.edu}

\begin{abstract}
Given an additively written abelian group $G$ and a set $X\subseteq G$, 
we let $\mathsf{D}(X)$ denote the Davenport constant of $X$, namely the largest non-negative integer $n$ for which 
there exists a sequence $x_1, \dots, x_n$ of elements of $X$ such that $\sum_{i=1}^n x_i =0$ and $\sum_{i \in I} x_i \ne 0$ 
for each non-empty proper subset $I$ of $\{1, \ldots, n\}$. 
In this paper, we mainly investigate the case when $G$ is $\mathbb{Z}^2$ and $\mathbb{Z}^3$, and $X$ is a discrete Euclidean ball. 
An application to the classical problem of estimating the Davenport constant of a box -- a product of intervals of integers -- 
is then obtained.
\end{abstract}


\maketitle

\section{Introduction}

Let $G$ be an additively written abelian group. 
Given $X \subseteq G$, we consider the free abelian monoid over $X$.
Its elements will be called sequences.
Let $S = x_1, \dots, x_n$ be a sequence of elements of $X$. 
We shall say that
the $x_i$'s are elements of $S$ or, simply, are in $S$ (that is, we identify sequences and multisets). 
The sequences considered here are unordered.
We call \textit{support} the set of (distinct) elements of $G$ appearing in $S$. 
We call $n$ the {\em length} of $S$, which we denote by $\|S\|$.
We notice that the support of $S$ can be drastically smaller than its length.
Finally a sequence $S = x_1, \dots, x_n$ will be called a {\em zero-sum sequence}
whenever $\sum_{i=1}^n x_i = 0$. A zero-sum sequence will be called \textit{minimal} 
if $\sum_{i \in I} x_i \ne 0$ for every non-empty proper subset $I$ of $\{1, \ldots, n\}$. 

For $G$ an abelian group, the study of its minimal zero-sum sequences and their combinatorial properties is part of 
what is called {\em zero-sum theory}, a subfield of additive combinatorics with applications to group theory, graph theory, 
Ramsey theory, discrete geometry and factorization theory. Since the 1960s, the theory of these invariants has highly developed in several directions.
The interested reader is referred to the surveys \cite{GG06, GB} or \cite{Ge-HK06a} (see in particular Chapters 5, 6, and 7).

One of the earliest questions in this area, and maybe one of the most important, is concerned with the \emph{Davenport constant}, which has become the prototype of algebraic invariants of combinatorial flavour.
More precisely, the Davenport constant of an abelian group $G$,
denoted by $\DD (G)$, is defined as the maximal length of a minimal zero-sum sequence over $G$.
A lot is known on this invariant (precise values and bounds depending on the case studied): see the references above for precise results.

The study of certain generalisations of the Davenport constant to subsets of not necessarily finite abelian groups 
is of great interest for its applications to factorization theory, where it can be used to describe
the arithmetic of certain monoids, see \cite{GGSS}, Chapter 3.4 of \cite{Ge-HK06a} and \cite{Grynkie22}. 
For instance, in \cite{BaethG}, Baeth and Geroldinger ask specifically for a study
of the Davenport constant of what we call a box, that is a product of intervals of integers.

Let again $G$ be a not necessarily finite abelian group and $X$ be any subset of it.  We define its Davenport constant, 
which we denote by $\DD(X)$, as the maximal length of a minimal zero-sum sequence over $X$;
this variant was first introduced by van Emde Boas in \cite{vEB}, where it is denoted by $\mu(G, X)$.
It is trivial but worth remarking that, in general,  
it can happen that $\DD(X)$ is finite and yet we can build arbitrarily long sequences of elements of $X$ 
with no non-empty zero-sum subsequence.
Also, it is immediate that $\DD( X) \leq \DD( G)$: 
this inequality is in general strict and it is well possible that $\DD( X)$ is finite while $\DD( G)$ is not.

Being given $m$ and $d$ two positive integers, what is called a {\em box} in what follows is a product of intervals 
of the form $\llbracket -m,m \rrbracket ^d$ 
(as usual, the notation $\llbracket a,b \rrbracket$ (for two integers $a \leq b$) stands for 
the set, called {\em interval}, of integers $i$ satisfying $a \leq i \leq b$ ; we shall use this notation or 
$\{ 1, \dots, n \}$ depending on the context). 

The study of the Davenport constant of a box is already documented.
In the case where $d=1$, one has $\DD (\llbracket -1,1 \rrbracket )=2$ and it is almost folklore that 
$$
 \DD (\llbracket -m,m \rrbracket ) =2m-1 
$$
for every integer $m \geq 2$ (see \cite{Lambert} for an early proof and \cite{DZ2,DZ,DGS, PT} for rediscoveries and further development). 
Recently, the authors of the present paper obtained the exact value of the Davenport constant for any interval $\llbracket -m,M \rrbracket$, 
where $m,M$ are positive integers \cite{GP25}: the proof is considerably more sophisticated. 
When $d =2$, it is known \cite{PT} that $ \DD(\llb -1,1\rrb^2)=4$ and that
\begin{equation}
\label{DDbox}
(2m-1)^2 \le \DD(\llb -m,m\rrb^2) \le (2m+1)(4m+1)
\end{equation}
for every integer $m \geq 2$,
a better lower bound being obtained in \cite{GP3}. 
And in general, for $d\ge 3$ and every integer $m \geq 2$, the best known bounds \cite{PT} are
\begin{equation}
\label{DDbox3et plus}
(2m-1)^d \le \DD(\llb -m,m\rrb^d) \le \left( 2 \left(d+\frac1d -1 \right) m+1 \right)^d,
\end{equation}
with a better lower bound, asymptotic to $16m^3$ as $m$ tends to infinity, being known \cite{GP3} when $d=3$. 

\section{New results on balls and boxes}

The main goal of the present paper is to introduce and study $\DD(X)$ in the case where 
$X$ is a discrete Euclidean ball (with radius $m$, an integer) of $\ZZ^d$ ($d$ being a positive integer), 
that is, the following set of integral points:
$$
X= \Bcal^{(d)}_m =\{ (u_1,\dots, u_d) \in \ZZ^d : u_1^2+ \cdots +u_d^2 \leq m^2 \}.
$$

Let $m$ be a positive integer. If $m=1$, then $\DD (\Bcal^{(d)}_1)=2$ for all $d \ge 1$, since it is easily seen that any minimal zero-sum sequence over $\Bcal^{(d)}_1$ of length at least $2$ must be of the form $u,-u$ for some $u \in \Bcal^{(d)}_1 \setminus \{0\}$. Now, let us assume that $m \ge 2$. In view of the double inclusion
$$
\left\llb -\left[ \frac{m}{\sqrt{d}} \right], \left[ \frac{m}{\sqrt{d}} \right] \right\rrb^d  \subseteq \Bcal^{(d)}_m \subseteq \llb -m,m \rrb^d,
$$
(as usual, $[x]$ denotes the integral part of a real $x$) we have
$$
\DD \left( \left\llb -\left[ \frac{m}{\sqrt{d}} \right], \left[ \frac{m}{\sqrt{d}} \right] \right\rrb^d \right) \leq \DD (\Bcal^{(d)}_m) \leq \DD ( \llb -m,m \rrb^d),
$$
which, in view of \eqref{DDbox3et plus}, gives for $d \geq 3$
\begin{equation}
\label{dimdgen}
\left( 2 \left[ \frac{m}{\sqrt{d}} \right] -1 \right)^d \leq \DD (\Bcal^{(d)}_m) \leq \left( 2 \left(d+\frac1d -1 \right) m+1 \right)^d,
\end{equation}
which readily yields $m^d$ as the right order of magnitude of $\DD (\Bcal^{(d)}_m)$. 
When $d=2$, it follows from \eqref{DDbox} that
$$
\left( 2 \left[ \frac{m}{\sqrt{2}} \right] -1 \right)^2 \le \DD( \Bcal^{(2)}_m ) \le (2m+1)(4m+1),
$$
which gives asymptotically, as $m \to + \infty$,
\begin{equation}
\label{B2}
2 \lesssim  \frac{ \DD (\Bcal^{(2)}_m)}{m^2} \lesssim 8,
\end{equation}
where $u_m \lesssim v_m$ means that $\liminf_{m \rightarrow + \infty} v_m / u_m \geq 1$.

We first study more precisely the case of dimension 2 and prove the following theorem.

\begin{theorem}
\label{thmdisk}
Let $m \geq 5$ be an integer, then 
$$
\frac{64}{25} m^2 -23m +51 \leq \DD (\Bcal^{(2)}_m) \leq \frac{5 \pi}{4} \left( m+ \frac{4}{\sqrt{5}} \right)^2.
$$
Moreover, asymptotically as $m \to + \infty$,
$$
\frac{3 \sqrt{3}}{2} m^2 + O(m^{1.1}) \leq \DD (\Bcal^{(2)}_m) \leq \frac{5 \pi}{4} m^2 + O(m).
$$
\end{theorem}

To have an insight on the quality of these bounds, we compute that $64/25 =2.56$ 
and $3 \sqrt{3} /2 \simeq 2.598\dots$ while $5 \pi /4 \simeq 3.926\dots$ which improves on both 
sides of \eqref{B2}.
For this result, error terms were given. Even though we will not always give them explicitly in the rest of the paper, making the error terms precise will be most of the time possible.

To investigate the Davenport constant, we introduce the quantity
$
\DD^{(k)} (X)  
$
defined in the same way as $\DD (X)$, with the additional restriction on the maximum that only sequences
having a support of size $k$ are considered.

For the sake of argument, assume the ground set $X \subseteq \ZZ^d$ to be ``sufficiently regular".
For instance, one could think of $X$ as a ball of (a large enough) radius $m$, for some norm over $\mathbb{R}^d$.

On the one hand, it appears likely that a minimal zero-sum sequence of maximal length over $X$ should be full-dimensional: for, if $S=(s_1,0), \dots, (s_l,0)$ is a minimal zero-sum sequence over $X \cap( \ZZ^{d-1} \times \{0 \})$, a small perturbation $\varepsilon_i \in  \ZZ^{d-1}$ could be found for some $i \in \{1,\dots,l\}$ so that $u_i=(s_i-\varepsilon_i,1)$ and $v_i=(\varepsilon_i,-1)$ both belong to $X$. As a result, the new longer sequence over $X$ obtained from $S$ by deleting $(s_i,0)$ and adding $u_i$ and $v_i$ would also be minimal and zero-sum, a contradiction.
We thus expect a minimal zero-sum sequence of maximal length over such an $X$ to have a support of size at least $d+1$.
Indeed, seen as vectors of some finite-dimensional vector space over $\R$, the elements of a zero-sum sequence
are linearly dependent; this implies that the dimension of the vector space generated by the elements of a zero-sum sequence cannot
be larger than the size of its support minus one.

On the other hand, for statistical reasons, it is likely that the support of a minimal zero-sum sequence over $X$ should not be too large.
Indeed, the bigger the support, the bigger the mixing in the sums we can form and
thus the bigger the expected number of distinct sums
(to get a feeling of this philosophy, see for instance related results on subset sums \cite{Freiman,Plagne}).

For these reasons, being able to compute quantities like $\DD^{(d+1)} (X)$ appears to be interesting, since
\begin{equation}
\label{mino2}
\DD (X) \geq  \DD^{(d+1)}(X),
\end{equation}
this lower bound being hopefully optimal.
For instance, finding a lower bound for $\DD^{(3)} (\Bcal^{(2)}_m)$ is precisely how we will obtain the lower bound in Theorem \ref{thmdisk}. 
In addition, our proof of Theorem \ref{thmdisk} can be complemented by proving that our construction of a long minimal zero-sum sequence with 
a support of size 3 is asymptotically optimal, so that we obtain the following theorem.

\begin{theorem}
\label{theooptimal}
Asymptotically, as $m \to + \infty$,
$$
\DD^{(3)} (\Bcal^{(2)}_m) \sim \frac{3 \sqrt{3}}{2} m^2.
$$
\end{theorem}

An interesting remark here is that (asymptotically) the three vectors forming the support of the sequence leading 
to this asymptotic coefficient are the vertices of an equilateral triangle. 
This will be discussed later (see Sections \ref{sectthmdisk} and \ref{prtheopt}).

Another point of interest of our study of the Davenport constant of discrete Euclidean balls is that 
it leads to good upper bounds for the Davenport constant of boxes, for instance in dimension 2, 
which was our starting point. 
As a corollary of Theorem \ref{thmdisk} together with the inequality 
$\DD (\llbracket -m,m \rrbracket ^2)  \leq \DD (\Bcal^{(2)}_{ \lceil \sqrt{2}m \rceil })$, 
we obtain the following upper bound for the Davenport constant of a two-dimensional box. 

\begin{theorem}
Asymptotically as $m \to + \infty$,
$$
 \DD (\llbracket -m,m \rrbracket ^2)  \leq  \frac{5 \pi}{2} m^2 + O(m).
 $$
 \end{theorem}

This theorem improves the best upper bound known so far \eqref{DDbox}, since $5 \pi/2 <7.854 <8$.

We continue our investigations with a result in the three-dimensional case.

\begin{theorem}
\label{thmd3}
Asymptotically as $m \to + \infty$,
$$
 \frac{16}{3 \sqrt 3}  m^3 \lesssim \DD (\Bcal^{(3)}_m) \lesssim \frac{1372}{81}\ \pi\ m^3.
$$
\end{theorem}

The constant appearing in the lower bound is equal to $3.079\dots$ and 
improves on that of  \eqref{dimdgen} by a multiplicative factor 2, 
while the constant on the right-hand side is about $53.213\dots$ 
which is significantly better than that of \eqref{dimdgen} (namely $101.630\dots$).

Using a similar approach as in Section \ref{prtheopt} for the proof of Theorem \ref{theooptimal}, we will show 
in Section \ref{optindim33} how to adapt our argument in dimension 3 and prove the following result (see \cite{GP3} for an application to boxes).

\begin{theorem}
\label{thmd4}
Asymptotically, as $m \to + \infty$,
$$
\DD^{(4)} (\Bcal^{(3)}_m) \sim  \frac{16}{3 \sqrt 3} m^3.
$$
\end{theorem}

As in dimension 2, the optimal choice in dimension $d=3$ for the elements of the support of a minimal zero-sum sequence 
with maximal length is close to a regular $(d+1)$-hedron, a tetrahedron. 

In general, we conjecture that using an optimal distribution of the $d+1$ points on the $(d-1)$-sphere in dimension $d$, 
and charging these points adequately leads to the value of $\DD^{(d+1)} (\Bcal^{(d)}_m)$ provided that $m$ is large enough in terms of $d$.
It also leads to the following conjecture on the Davenport constant of the ball in any dimension.

\begin{conjecture}
\label{conj78}
Let $d\geq 2$ be an integer. Then, for $m \ge \sqrt{d}$, one has
$$
\DD (\Bcal^{(d)}_m) = \DD^{(d+1)} (\Bcal^{(d)}_m) \underset{m \to + \infty}{\sim}  \left( \frac{(d+1)^{d+1}}{d^d} \right)^{1/2} m^d.
$$
\end{conjecture}

The cases $d=1,$ 2 or 3 of the second part of the conjecture are precisely the cases we treated above. 
Also, note that in the case of the box, one could be tempted to conjecture that $\DD (\llbracket -m,m \rrbracket ^d)/m^d$ has a limit too when $d$ is fixed and $m$ tends to $+\infty$ (for instance, a natural analogue of this problem in the torsion case is solved, see Theorem $1$ in \cite{G18}). However, and contrary to the case of the ball above, there is no guess as to what this limit could be in terms of $d$. 

\section{Proof of Theorem \ref{thmdisk}: lower bounds}
\label{sectthmdisk}

As for the lower bound, we start by considering the remainder $r$ of $m$ in the Euclidean division of $m$ by 5. 
We denote $m' =m-r$ which is divisible by 5. Let
$$
a = \frac{4m'}{5} \quad  {\rm and } \quad b = \frac{3m'}{5}.
$$
Notice that $a^2+b^2= m'^2$ thus the three points 
$
\begin{pmatrix}
0 \\
m' 
\end{pmatrix}, 
\begin{pmatrix}
-a \\
-(b-1) 
\end{pmatrix}   
$
and $\begin{pmatrix}
a-1 \\
-b 
\end{pmatrix}$ 
all belong to $\Bcal^{(2)}_{m'} \subseteq \Bcal^{(2)}_{m}$.

We consider the sequence $S_1$ consisting of 
$
\begin{pmatrix}
0 \\
m' 
\end{pmatrix}
$
repeated $ab+(a-1)(b-1)=2ab-a-b+1$ times, 
$
\begin{pmatrix}
-a \\
-(b-1) 
\end{pmatrix}
$ 
repeated $(a-1)m'$ times and 
$
\begin{pmatrix}
a-1 \\
-b
\end{pmatrix}
$ 
repeated $am'$ times.
It is a zero-sum sequence since
$$
(ab+(a-1)(b-1)) 
\begin{pmatrix}
0 \\
m' 
\end{pmatrix}
+ (a-1)m'
\begin{pmatrix}
-a \\
-(b-1) 
\end{pmatrix}
+ am'
\begin{pmatrix}
a-1 \\
-b
\end{pmatrix}
= 
\begin{pmatrix}
0 \\
0
\end{pmatrix}.
$$

It remains to show that $S_1$ is minimal. 
Assume 
\begin{equation}
\label{2coord}
k_1 \begin{pmatrix}
0 \\
m' 
\end{pmatrix}
+ k_2 
\begin{pmatrix}
-a \\
-(b-1) 
\end{pmatrix}
+ k_3 
\begin{pmatrix}
a-1 \\
-b
\end{pmatrix}
= 
\begin{pmatrix}
0 \\
0
\end{pmatrix}
\end{equation}
for non-negative integers $k_1,k_2,k_3$ that are not all zero and satisfy
\begin{equation}
\label{bound0}
k_1 \leq 2ab-a-b+1,\quad  k_2 \leq (a-1)m'\quad {\rm and }\quad  k_3 \leq am'.
\end{equation}
Considering the first coordinate in \eqref{2coord} shows that $k_2 a = k_3 (a-1)$ which implies, since $a$ and $a-1$ 
are coprime, that there exists an integer $k$ such that 
$$
k_2 = (a-1)k\quad {\rm and} \quad k_3 = ak.
$$
Considering the second coordinate in \eqref{2coord}  shows that 
\begin{equation}
\label{k1}
k_1 m' = (b-1) k_2+bk_3 = \left( (a-1)(b-1) + ab \right)k = (2ab-a-b+1)k.
\end{equation}

We observe that $\gcd (m', 2ab-a-b+1)=1$. Indeed, this assertion is equivalent to 
$$
\gcd (25m',  2 \cdot 4m'\cdot 3m' -  5\cdot 4m' -5\cdot 3m' +25)=\gcd (25m',  24 m'^2 -  35 m'+25)=25.
$$
Any prime $p$ dividing $25m'$ must divide $m'$, since $5$ divides $m'$. 
Therefore such a $p$ has to divide $24 m'^2 -  35 m'$ and finally $25$: this gives $p=5$. 
However, while $25$ divides $24 m'^2 -  35 m'+25$, $125$ doesn't due to the fact 
that the polynomial $4X^2+3X+1$ doesn't vanish mod $5$.
This proves our assertion.

Finally, with this coprimality result, equality \eqref{k1} shows that $(2ab-a-b+1) \mid k_1$ and $m' \mid k$ which in view of \eqref{bound0} shows that
$$
k_1 = 2ab-a-b+1,\quad k=m',\quad k_2 = (a-1)m'\quad {\rm and} \quad k_3 = am'.
$$
This proves our assertion on the minimality of the sequence $S_1$.

We finally obtain
$$
\DD (\Bcal^{(2)}_m)  \geq	\DD (\Bcal^{(2)}_{m'}) \geq  \|S_1 \|  = 2a(b+m')-a-b-m'+1  \geq 	 \frac{64}{25} m'^2 -\frac{12}{5} m' +1.
$$
Since this lower bound is an increasing function of $m'$ on positive integers and $m' =m-r$ with $0 \leq r \leq 4$ we 
obtain
$$
\DD (\Bcal^{(2)}_m) \geq \frac{64}{25} (m-4)^2 -\frac{12}{5} (m-4) +1 = \frac{64}{25} m^2 -\frac{572}{25} m + \frac{1289}{25}
$$
and our first result follows.

We now come to the second assertion in our theorem. 
Let $m$ be a large enough integer. 
We consider $p_m$ the largest prime $\leq 2 \lfloor  \sqrt{3}m/2 \rfloor-1$ 
and define $a_m = (1+p_m)/2$. We notice immediately that 
$$
a_m \sim \frac{p_m}{2} \sim \frac{\sqrt{3}}2 m
$$
since $p_m \sim  2 \lfloor  \sqrt{3}m/2 \rfloor \sim \sqrt{3}m$ 
in view of the known results on the gap between consecutive primes, see for instance \cite{BHP}. 
More precisely, these results imply that $p_m = \sqrt{3} m + O(m^{.55})$ which is by far enough 
to get a good error term in our estimate below.

We now define the following integer:
$$
b_m = \left\{ 
\begin{array}{ll}
\frac{m-1}{2} & {\rm if }\,\, m\,\, {\rm is\, odd},\\
\frac{m}{2} -1 &{\rm if }\,\, m \equiv 0 \pmod 4, \\
\frac{m}{2} -2 &{\rm if }\,\, m \equiv 2 \pmod 4.
\end{array}
\right.
$$
A simple computation shows that $\gcd(m, b_m)=1$, whatever the value of $m$.

We consider the sequence $S_2$ consisting of 
$
\begin{pmatrix}
0 \\
m
\end{pmatrix}
$ repeated $(2a_m-1)b_m$ times, 
$
\begin{pmatrix}
-a_m \\
-b_m 
\end{pmatrix}
$ 
repeated $(a_m-1)m$ times and 
$
\begin{pmatrix}
a_{m}-1 \\
-b_m 
\end{pmatrix}
$ 
repeated $a_m m$ times.
Its elements all belong to $\Bcal^{(2)}_m$ since the norm of the second and third vectors are less than or equal to
$$
\sqrt{a_m^2 + b_m^2} \leq \sqrt{ \lfloor  \sqrt{3}m/2 \rfloor^2 + \frac{m^2}{4}} \leq \sqrt{ \left( \frac{\sqrt{3}m}{2} \right)^2 + \frac{m^2}{4}} =m.
$$
It is a zero-sum sequence since
$$
(2a_m-1)b_m 
\begin{pmatrix}
0 \\
m
\end{pmatrix}
+ (a_m-1)m
\begin{pmatrix}
-a_m \\
-b_m 
\end{pmatrix}
+ a_m m
\begin{pmatrix}
a_{m}-1 \\
-b_m 
\end{pmatrix}
= 
\begin{pmatrix}
0 \\
0
\end{pmatrix}
.
$$

It remains to show that $S_2$ is minimal. 
Assume 
$$
k_1
\begin{pmatrix}
0 \\
m
\end{pmatrix}
+ 
k_2
\begin{pmatrix}
-a_m \\
-b_m 
\end{pmatrix}
+ 
k_3
\begin{pmatrix}
a_{m}-1 \\
-b_m 
\end{pmatrix}
= 
\begin{pmatrix}
0 \\
0
\end{pmatrix}
$$
for non-negative integers $k_1,k_2,k_3$ that are not all zero and satisfy
\begin{equation}
\label{bound1}
k_1 \leq (2a_m-1)b_m,\quad  k_2 \leq (a_m-1)m\quad {\rm and }\quad  k_3 \leq a_m m.
\end{equation}
As above, and since $\gcd (a_m, a_m -1)=1$, considering the first coordinate shows that there exists an integer $k$ such that 
$$
k_2 = (a_m-1)k\quad {\rm and} \quad k_3 = a_m k.
$$
Considering the second coordinate shows that 
\begin{equation}
\label{k12}
k_1 m  = (k_2+k_3) b_m = (2a_m-1)b_m k=p_m b_m k.
\end{equation}

We observe that, when $m$ is large enough, $\gcd (m, b_m p_m)=1$. Indeed when $m$ is large, $p_m >m$ 
since it is $\sim \sqrt{3}m$, thus $p_m$ being prime gives $\gcd (m,p_m)=1$;  
while $m$ and $b_m$ are always coprime as mentioned earlier.

Finally, using the fact that $\gcd (m, b_m p_m)=1$, equation \eqref{k12} yields $(2a_m-1) b_m \mid k_1$ and $m \mid k$ 
which in view of \eqref{bound1} shows that
$$
k_1 = (2a_m-1)b_m,\quad k=m,\quad k_2 = (a_m-1)m\quad {\rm and} \quad k_3 = a_m m.
$$
This proves our assertion on the minimality of the sequence $S_2$.

We finally obtain
$$
\DD (\Bcal^{(2)}_m)  \geq 	\|S_2 \|  =	 (2a_m-1)(b_m+m) 
			\sim 2a_m(b_m+m)
			\sim 2 \frac{\sqrt{3}}2 m \times \frac{3m}{2}
			\sim  \frac{3 \sqrt{3}}{2} m^2.		
$$
The error term in our result comes from the remark made after the definition of $a_m$ above.

\section{Proof of Theorem \ref{thmdisk}: the upper bound}
\label{useofsteinitz}

We proceed as in Lemma 5 of \cite{PT}. Let $k=\DD ( \Bcal^{(2)}_m  )$  and $u_1, \dots, u_k$ 
be a minimal zero-sum sequence of elements of $\Bcal^{(2)}_m $ attaining this bound.

Let us recall the theorem of Steinitz \cite{Stein13} who gave the first complete proof of the following result.

\begin{theorem}
\label{steinitzthm}
Let $d$ be a positive integer and $U \subseteq \RR^d$ be such that $0 \in U$. 
There exists a constant $c$ such that whenever $v_1, \ldots, v_n \in U$ and $v_1 + \cdots + v_n = 0$, 
there is a permutation $\pi$ of $\{ 1,\dots, n \}$ such that $v_{\pi(1)} + \cdots + v_{\pi(i)} \in c \cdot U$ 
for each $i \in \{ 1,\dots, n \}$.
\end{theorem}

In this statement, we used the notation $\alpha \cdot U$  for the $\alpha$-dilate of $U$, namely
$
\alpha \cdot U = \{\alpha u: u \in U\}$.

We shall call {\em the Steinitz constant of $U$} the infimum of all constants 
$c \in \mathbb \RR^+$ that can be taken in Theorem \ref{steinitzthm}, and denote it by $C(U)$.
Bergstr\"om \cite{Berg} and Banaszczyk \cite{Banasz87, Banasz90} proved that 
in the case of the unit ball $B_1$ (the unit disk) of $\R^2$ 
\begin{equation}
\label{stein}
C(B_1) = \sqrt{5}/2.
\end{equation}

Let $v_i = u_i/m$ for $1 \leq i \leq k$. 
The sequence $(v_i)_{1 \leq i \leq k}$ has all its elements in $B_1$ and is zero-sum. 
Moreover, Steinitz' result proves that there exists a permutation $\pi$ of $\{ 1,\dots,  k \}$ such that 
$v_{\pi(1)} + \cdots + v_{\pi(i)}$ belongs to the disk of radius $\sqrt{5}/2$ for all $i \in \{1, \dots, k\}$.

All the sums $u_{\pi(1)} + \cdots + u_{\pi(i)}$ must be distinct since otherwise 
we would find two distinct indices $l < l'$ such that 
$$
u_{\pi(1)} + \cdots + u_{\pi(l)} = u_{\pi(1)} + \cdots + u_{\pi(l')},
$$
that is $u_{\pi(l'+1)} + \cdots + u_{\pi(l)}=0$,
which would give a non-empty zero-sum subsequence of $(u_i)_{1 \leq i \leq k}$. 
This would contradict the fact that $(u_i)_{1 \leq i \leq k}$ is minimal as a zero-sum sequence.

Thus, $k$ is at most the number of integral points of $\sqrt{5}/2 \cdot B_1= B_{\sqrt{5}/2}$, 
the disk of radius $\sqrt{5}/2$.
This implies 
$$
\DD (\Bcal^{(2)}_m) \leq \pi \left( \frac{\sqrt5 m}2 +2 \right)^2 \sim \frac{5 \pi}4 m^2.
$$
For this estimate, we use the fact that the number of integral points of the disk $B_{\sqrt{5}/2\ m}$ is bounded from above by the area of the disk $B_{\sqrt{5}/2\ m+2}$. 
Indeed, we may injectively map each integral point in $B_{\sqrt{5}/2\ m}$ to the upper-right square 
of which it is a vertex. All these squares are contained in the disk $B_{\sqrt{5}/2\ m+2}$. 

This last remark illustrates a general principle (for which we will not go into details in the sequel), that we now state:
\smallskip

\textbf{Principle of approximation of the volume by the number of integral points.}
{\em The volume of a sufficiently nice and regular set is well approximated by the number of integral points it contains.}
\smallskip

We refer to \cite{EGH, F}. This general principle, albeit dangerous sometimes \cite{MO}, is valid in particular for balls, boxes, convex polyhedra,
at least in the sense that if such a set $S$ is given, one has asymptotically as $t$ tends to $+ \infty$,
$$
\vol (t\cdot S ) \sim {\rm N}( t \cdot S),
$$
where $\vol$ and ${\rm N}$ stand for the volume of and the number of integral points contained in the considered set.

\section{Proof of Theorem \ref{theooptimal}}
\label{prtheopt}

We start our proof with a general old-school lemma of classical geometry, the advantage being that things are visible and 
appear at first glance. We will use a more formal approach later in this paper.

In general, for $d$ a positive integer, what we mean by {\em positively dependent vectors} is a set of vectors $w_1,\dots , w_{d+1} $ in $\R^d$ 
which span $\R^d$ and such that there are positive real numbers $a_1, \dots , a_{d+1}$ such that 
\begin{equation}
\label{posdep}
a_1w_1 + \cdots + a_{d+1} w_{d+1} = 0.
\end{equation} 
This is tantamount to asking that there is no closed half-space (half-$\R^d$) containing all vectors.

Notice that from $d+1$ positively dependent vectors $w_1,\dots , w_{d+1} $ in $\R^d$, any choice of $d$ of them forms a basis: 
otherwise there would be a way to select $d$ vectors spanning a $(d-k)$-dimensional subspace of $\R^d$ for some $k>0$.
By \eqref{posdep} all $d+1$ vectors would lie in this $(d-k)$-dimensional vector space, thus not spanning $\R^d$.

In this section, we restrict ourselves to the case $d=2$ and consider three positively dependent vectors $w_1,w_2,w_3$ in $\R^2$.
This is tantamount to asking that there is no closed half-plane containing the three vectors. 
In this situation, the convex hull of $w_1,w_1+w_2,w_2,w_2+w_3,w_3,w_1+w_3$ is a hexagon 
containing $0$ by \eqref{posdep}, as illustrated by Figure 1.

\begin{figure}[h!]
\begin{center}
\begin{tikzpicture}[>=stealth,scale=0.4]
\foreach \point in {(0,0),(-4,3),(3,2),(-1,-5)} 
\fill \point circle (0.5pt);
\draw [->,thick] (0,0) -- (0,5) node [anchor=south] {$w_1$};
\draw [->,thick] (0,0) -- (-4,-2) node [anchor=north east] {$w_2$};
\draw [->,thick] (0,0) -- (3,-3) node [anchor=north west] {$w_3$};
\draw (0,5) -- (-4,3) node [anchor=south east] {$w_1+w_2$} -- (-4,-2) -- (-1,-5) node [anchor=north] {$w_2+w_3$} -- (3,-3) -- (3,2) node [anchor=west] {$w_1+w_3$} -- (0,5);
\end{tikzpicture}
\smallskip
\end{center}
\caption{A representation of $H(w_1,w_2,w_3)$.}
\end{figure}
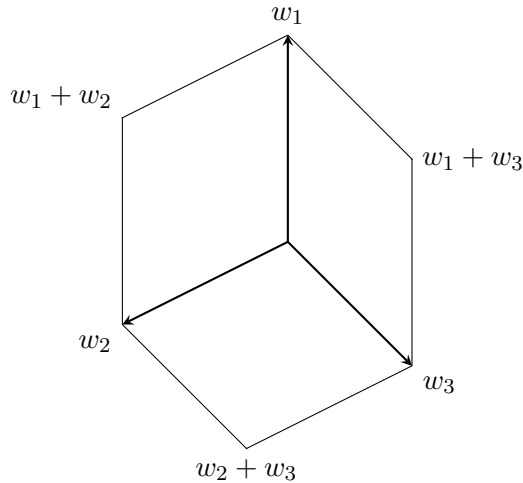

\begin{lemma} 
\label{wC1}
Let $w_1,w_2,w_3$ be positively dependent vectors in $\R^2$ and let $H(w_1,w_2,w_3)$ be the convex hull of $w_1,w_1+w_2,w_2,w_2+w_3,w_3,w_1+w_3$.
Set the three parallelograms $C_1=\{\alpha w_2 + \beta w_3 : 0 \le \alpha,\beta \le 1\}, C_2=\{\alpha w_1 + \beta w_3 : 0 \le \alpha,\beta \le 1\}$ and 
$C_3=\{\alpha w_1 + \beta w_2 : 0 \le \alpha,\beta \le 1\}$. 
Then, 
$$
H (w_1,w_2,w_3)=\displaystyle\bigcup^3_{i=1} C_i = \displaystyle\bigcup^3_{i=1} (w_i+C_i).
$$
\end{lemma}

\begin{proof}
The situation is as in Figure 2.
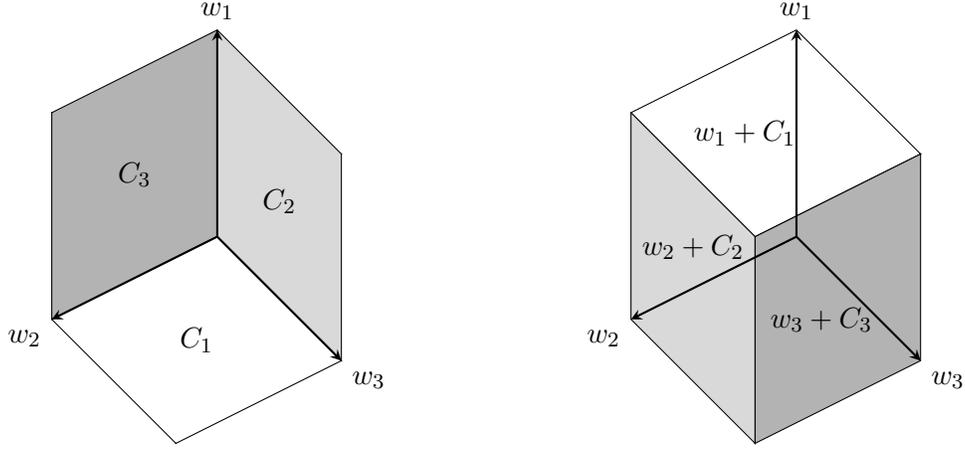
\begin{figure}[h!]
\begin{center}
\begin{tikzpicture}[>=stealth,scale=0.441]
\foreach \point in {(0,0),(-4,3),(3,2),(-1,-5)} 
\fill \point circle (0.5pt);
\draw[fill=black!30] (0,5) -- (-4,3) -- (-4,-2) -- (0,0) -- (0,5);
\draw[fill=black!15] (0,5) -- (3,2) -- (3,-3) -- (0,0) -- (0,5);
\draw (0,0) -- (-4,-2) -- (-1,-5) -- (3,-3) -- (0,0) ;
\draw [->,thick] (0,0) -- (0,5) node [anchor=south] {$w_1$};
\draw [->,thick] (0,0) -- (3,-3) node [anchor=north west] {$w_3$};
\draw [->,thick] (0,0) -- (-4,-2) node [anchor=north east] {$w_2$};
\draw (-0.5,-2.5) node {{\small $C_1$}};
\draw (1.5,0.85) node {{\small $C_2$}};
\draw (-2,1.5) node {{\small $C_3$}};
\foreach \point in {(14,0),(10,3),(17,2),(13,-5),(13,0)} 
\fill \point circle (0.5pt);
\draw (10,3)  -- (13,0) -- (17,2) -- (14,5) -- (10,3); 
\draw[fill=black!15] (10,3) -- (10,-2) -- (13,-5) -- (13,0) -- (10,3); 
\draw[fill=black!30] (13,0) -- (13,-5) -- (17,-3) -- (17,2) -- (13,0); 
\draw [->,thick] (14,0) -- (14,5) node [anchor=south] {$w_1$};
\draw [->,thick] (14,0) -- (10,-2) node [anchor=north east] {$w_2$};
\draw [->,thick] (14,0) -- (17,-3) node [anchor=north west] {$w_3$};
\draw (12.25,2.95) node {{\small $w_1+C_1$}};
\draw (11.5,-0.09) node {{\small $w_2+C_2$}};
\draw (14.6,-2.5) node {{\small $w_3+C_3$}};
\end{tikzpicture}
\smallskip
\end{center}
\caption{Two decompositions of the hexagon $H(w_1, w_2, w_3)$.}
\end{figure}
The first assertion of the lemma is almost immediate. 
For if we complete $w_1$ by a vector $e_2$ such that $(w_1, e_2)$ becomes our canonical orthogonal basis, and up to renaming our vectors, 
we may assume that $w_2$ has an $e_2$-coordinate which is positive (otherwise, by \eqref{posdep}, it must be $w_3$), 
and $w_3$ a negative $e_2$-coordinate. 
By rotating counterclockwise from $w_1$, one thus meets the vertices in that order: 
$w_1,w_1+w_2,w_2,w_2+w_3,w_3,w_1+w_3$. 
Each of the $C_i$'s appears to be the intersection of $H$ with an angular sector, for instance 
$$
C_1 = H(w_1, w_2, w_3) \cap S(w_2, w_3)
$$
where $S(w_2, w_3)$ is the angular sector delimited by the origin and the two vectors $w_2$ and $w_3$.

Let us now consider the second assertion. 
We start by applying the first identity to the set of vectors $-w_1,-w_2,-w_3$ which are positively dependent as well. 
We obtain
$$
H (- w_1, -w_2, -w_3)=\displaystyle\bigcup^3_{i=1} C'_i
$$
where $C'_1=\{-\alpha w_2 - \beta w_3 : 0 \le \alpha,\beta \le 1\}, C'_2=\{- \alpha w_1 - \beta w_3 : 0 \le \alpha,\beta \le 1\}$ and 
$C'_3=\{- \alpha w_1 - \beta w_2 : 0 \le \alpha,\beta \le 1\}$. We now translate this identity by the vector $w_1+w_2+w_3$. 
This gives
\begin{equation}
\label{wH}
(w_1+w_2+w_3) + H (- w_1, -w_2, -w_3)=\displaystyle\bigcup^3_{i=1} (w_1+w_2+w_3 +C'_i ).
\end{equation}
We consider the left-hand side first (the notation $\Conv$ in this chain of equalities is for the convex hull of a set of elements):
\begin{eqnarray*}
	&& (w_1+w_2+w_3) + H (- w_1, -w_2, -w_3) \\
	& = & (w_1+w_2+w_3) + \Conv (-w_1,-w_1-w_2, -w_2,-w_2-w_3,-w_3,-w_1-w_3) \\
						 		& = & \Conv (w_2+w_3, w_3, w_1+w_3, w_1, w_1+w_2, w_2) \\
								& = & H(w_1, w_2, w_3).
\end{eqnarray*}
Let us now deal with the $(w_1+w_2+w_3 +C'_i )$'s, by considering the case $i=1$: 
\begin{eqnarray*}
(w_1+w_2+w_3) + C'_1 & = & (w_1+w_2+w_3) + \{-\alpha w_2 - \beta w_3 : 0 \le \alpha,\beta \le 1\}\\
						 		& = & w_1 +  \{(1-\alpha) w_2 + (1- \beta) w_3 : 0 \le \alpha,\beta \le 1\}\\
								& = & w_1 +  \{\alpha w_2 + \beta w_3 : 0 \le \alpha,\beta \le 1\} = w_1 + C_1.
\end{eqnarray*}
Thus \eqref{wH} gives
$$
H(w_1, w_2, w_3) = \displaystyle\bigcup^3_{i=1} (w_i+C_i ).
$$

\end{proof}

From this lemma, we deduce an important corollary of an algorithmic nature.

\begin{corollary}
\label{corosuiteH}
Let $n \ge 3$ and $u_1,\dots, u_n$ be a minimal zero-sum sequence  in $\mathbb{Z}^2$ 
having a support of size $3$, say $\{w_1,w_2,w_3\}$.
Then, $n$ is at most the number of integral points in $H(w_1,w_2,w_3)$.
\end{corollary}

\begin{proof}
Since $\sum^n_{i=1} u_i=0$, the elements $w_1,w_2,w_3$ must be positively dependent. 
We shall prove by induction that there is a permutation $\sigma$ of $\{1,\dots, n\}$ such that 
$s_j =\sum^j_{i=1} u_{\sigma(i)} \in H(w_1,w_2,w_3)$ for every $j \in \{1,\dots,n\}$.

We choose $\sigma (1)=1$. Thus $s_1 =  u_{\sigma(1)}  = u_{1} \in H(w_1,w_2,w_3)$.

Suppose now that for some positive integer $k<n $, the indices $\sigma (1),\dots, \sigma (k)$ are already chosen so that 
$s_j \in H(w_1,w_2,w_3)$ for all $j \in \{1,\dots, k\}$. 
By Lemma \ref{wC1}, there is a value $i_0$ in $\{1,2,3\}$ such that $s_k \in C_{i_0}$. 
We choose $\sigma (k+1)\in \{1,\dots, n\} \setminus \{ \sigma (1),\dots, \sigma (k) \}$ such that $u_{\sigma (k+1)}=w_{i_0}$. 

Notice that there must be such a value for $\sigma (k+1)$ otherwise all the remaining values in the sequence $u_1,\dots, u_n$
(those with index in $\{1,\dots, n\} \setminus \{ \sigma (1),\dots, \sigma (k) \}$) 
are equal to one of the two other values $w_j$ or $w_{j'}$ (with $\{j,j'\} = \{1,2,3\} \setminus \{i_0\}$). 
Thus 
$$
0 = \sum^n_{i=1} u_i = \sum^k_{i=1} u_{\sigma (i)} + \sum_{i \in \{1,\dots, n\} \setminus \{ \sigma (1),\dots, \sigma (k) \}} u_i 
\quad \in \quad C_{i_0} + \big(  (\NNb \cdot w_j + \NNb \cdot w_{j'} ) \setminus \{ 0 \} \big),
$$
(the term $ \setminus \{ 0 \} $ being due to the fact that we suppose $k<n$), which is a contradiction.

Then, by Lemma \ref{wC1},
$$
s_{k+1} = s_k +u_{\sigma (k+1)}=s_k+w_{i_0} \in w_{i_0}+C_{i_0} \subseteq \bigcup^3_{i=1} (w_i+C_i) = H(w_1,w_2,w_3),
$$
and the induction step is completed.

Since our sequence has no non-trivial zero-sum subsequence, 
it follows that $s_1,\dots,s_n$ are distinct elements of $H(w_1,w_2,w_3)$, which gives the desired result.
\end{proof}

Before applying this result to prove the upper bound of Theorem \ref{theooptimal}, we must take care of the case
when the support of our sequence consists of three collinear vectors. 
But in this case, up to a rotation, we are in the situation of a one-dimensional sequence in which case we know 
that the maximal length of a minimal zero-sum sequence is at most $2m-1$, a much smaller bound than the one we obtain thereafter
(which grows like $m^2$).

We now investigate the case where the support (of size 3) of our minimal zero-sum sequence spans $\R^2$. 
The three vectors composing the support of our sequence must be positively dependent in view of the fact that the sequence is zero-sum.
By Corollary \ref{corosuiteH}, the value of $\DD^{(3)} (\Bcal^{(2)}_m)$ is bounded from above by the maximal number of integral points in a hexagon of the form $H(w_1,w_2,w_3)$ with the $w_i$'s in $\Bcal^{(2)}_m$. Approximating this number of integral points 
by its area is valid, in view of what we called the Principle of approximation of the volume by the number of integral points
(see Section \ref{useofsteinitz}).

We now face the question of maximising the area of such a hexagon as $H(w_1,w_2,w_3)$ when 
$w_1,w_2,w_3 \in \Bcal^{(2)}_m$.
 
In order to investigate this problem, denote the angle between $w_1$ and $w_2$ 
by $\alpha \in ]0, \pi[$ and the one between $w_2$ and $w_3$ by $\beta \in ]0, \pi[$. 
The total area $S(\alpha, \beta )$ of the hexagon $H(w_1,w_2,w_3)$ is the sum of the areas of the three parallelograms 
denoted by $C_1, C_2$ and $C_3$ in our previous notation, which are easily estimated and gives a total area of at most
$$
m^2 \big( \sin \alpha + \sin \beta + \sin \left( 2 \pi - (\alpha + \beta)\right) \big) = m^2 \big( \sin \alpha + \sin \beta - \sin \left( \alpha + \beta\right) \big)
$$
since the angle between $w_3$ and $w_1$ is $2 \pi - (\alpha + \beta)$.
Maximising this two-variable function is an easy task which leads to $\alpha= \beta=2\pi/3$ 
and consequently $2 \pi - \alpha - \beta = 2\pi/3$ as well. 
We obtain for $w_1,w_2,w_3$ an equilateral triangle so that $H(w_1,w_2,w_3)$ is a regular hexagon of side $m$,
the area of which is well known to be $(3\sqrt{3}/2)\ m^2$.

All in all, this gives
$$
\DD^{(3)} (\Bcal^{(2)}_m) \lesssim \frac{3\sqrt{3}}{2}m^2.
$$
This proves Theorem \ref{theooptimal}, since we already proved in Section \ref{sectthmdisk} that this bound is asymptotically attained.

Let us give a geometrical interpretation of this result. 
A look at Figure 2 shows that our hexagon can be seen as the two-dimensional projection of a cube, 
its three visible faces being what we called $C_1$, $C_2$ and $C_3$ on the left part of Figure 2.
If we now draw the sets $w_i+C_i$, as in the right part of Figure 2, these three translates appear 
to be the hidden faces (each of which is parallel to a visible face). 
Seen in this way, our maximisation problem becomes the following one: 
what is the maximum area of the projection of a cube's visible part?
It turns out that this problem is already documented \cite{McM} and that the answer is: 
this area is maximised for the cube projected in the direction of a grand diagonal, 
which corresponds to $w_1+w_2+w_3=0$ and gives the regular hexagon we already encountered.

To finish, notice that the area of the appearing hexagon is equal to twice the one of the equilateral triangle, a remark that will be useful later on.

\section{Proof of Theorem \ref{thmd3}}
\label{sectionthmd3}

The main part of our work deals with the lower bound in the theorem.

For $m$ large enough, we define $n_m$, $c_m$, $p_m$, $q_m$ and $r_m$ five primes being maximal with respect to the following constraints 
$$
n_m \leq m, \quad c_m \leq \frac{m}{3},\quad p_m \leq \frac{2\sqrt 2}{3} m,\quad q_m \leq \frac{\sqrt 2}{3} m\quad  {\rm and}\quad 
r_m <\frac{\sqrt{3}\ p_m}{2}  \leq \frac{\sqrt 2}{\sqrt3} m.
$$
As in the proof given in Section \ref{sectthmdisk}, and due to the fact that the gap between consecutive primes is small enough \cite{BHP}, we have
$$
n_m \sim m,\quad c_m \sim \frac{m}{3},\quad p_m \sim \frac{2\sqrt 2}{3} m\quad 
q_m \sim \frac{\sqrt 2}{3} m\quad {\rm and}\quad r_m \sim \frac{\sqrt{2}}{\sqrt{3}} m.
$$

For the lower bound, we use the sequence $S$ consisting of  the vector
$
\begin{pmatrix}
0 \\
0 \\
n_m
\end{pmatrix}
$ repeated $(2r_m-1)(q_m +p_m)c_m$ times, 
$
\begin{pmatrix}
0\\
p_m\\
- c_m \\ 
\end{pmatrix}
$ 
repeated $(2r_m-1)q_m n_m$ times, 
$
\begin{pmatrix}
r_{m}\\
- q_m \\ 
 -  c_m \\ 
\end{pmatrix}
$ 
repeated $(r_m-1)p_m n_m$ times and
$
\begin{pmatrix}
-(r_{m}-1)\\
- q_m \\ 
 -  c_m \\ 
\end{pmatrix}
$
repeated $r_m p_m n_m$ times.

Due to our choice for $n_m, p_m, q_m$ and $r_m$, this sequence has all its elements 
in $\Bcal^{(3)}_m$. It is a zero-sum sequence as can be checked by a direct computation as above.

It remains to show it is minimal. 
Assume 
$$
k_1
\begin{pmatrix}
0 \\
0 \\
n_m
\end{pmatrix}
+
k_2
\begin{pmatrix}
0\\
p_m\\
-  c_m \\ 
\end{pmatrix}
+
k_3
\begin{pmatrix}
r_{m}\\
- q_m\\ 
 -  c_m \\ 
\end{pmatrix}
+
k_4
\begin{pmatrix}
-(r_{m}-1)\\
- q_m\\ 
 -  c_m \\ 
\end{pmatrix}
=
\begin{pmatrix}
0 \\
0 \\
0
\end{pmatrix}
$$
for non-negative integers $k_1,k_2,k_3,k_4$ that are not all zero and satisfy
$$ 
k_1 \leq  (2r_m-1)(q_m +p_m)c_m,\  k_2 \leq  (2r_m-1)q_m n_m,\  k_3 \leq (r_m-1)p_m n_m\   {\rm and}\  k_4 \leq r_m p_m n_m. 
$$

As before, and by coprimality, the first coordinate shows that there exists an integer $k$ such that 
$$
k_3 = (r_m-1)k\quad {\rm and} \quad k_4 = r_m k.
$$
Considering the second coordinate shows that 
\begin{equation}
\label{k12x}
k_2 p_m  = (k_3+k_4) q_m = (2r_m-1) q_m k.
\end{equation}

We observe that $\gcd (p_m, (2r_m-1)q_m)=1$. 
Indeed when $m$ is large, $2r_m-1 <  \sqrt{3} p_m$ and $\sim \sqrt{3} p_m$ thus being  both striclty larger than $p_m$ and 
strictly smaller than $ 2 p_m$ ; on the other hand, in view of their respective asymptotic behaviour,  $q_m < p_m$.
We conclude by the primality of $p_m$.
With \eqref{k12x}, this shows that $(2r_m-1) q_m \mid k_2$ and $p_m \mid k$. 
Thus there exists an integer $l$ such that $k_2 = (2r_m-1)q_m l$ and $k=p_m l$.

Projecting on the third coordinate now shows that
$$
k_1 n_m = (k_2+k_3+k_4) c_m = \left( (2r_m-1)q_m + (2r_m-1) p_m \right) c_m l = (2r_m-1)(q_m +p_m)c_m l.
$$

Again, we notice that the terms $n_m$ and $(2r_m-1)(q_m +p_m)c_m$ must be coprime. 
Indeed, $2r_m-1 \sim (2\sqrt2 / \sqrt 3) m$ and $n_m \sim m $ thus, when $m$ is large, $n_m < 2r_m-1  < 2n_m$. 
Since $n_m$ is a prime, these two integers have to be coprime. 
The same holds for $q_m+p_m \sim \sqrt 2 m$ and $c_m \sim m/3$.

Therefore, $(2r_m-1)(q_m +p_m)c_m \mid k_1$ and $n_m \mid l$, so
$k_1 \geq (2r_m-1)(q_m +p_m)c_m$ and $l \geq n_m$, which implies that 
$$
k_2 \geq (2r_m-1)q_m n_m,\quad k_3 \geq (r_m-1)p_m n_m\quad {\rm and} \quad k_4 \geq r_m p_m n_m.
$$
This proves our assertion on the minimality of the sequence $S$.

We finally obtain
\begin{eqnarray*}
\DD^{(4)} (\Bcal^{(3)}_m) & \geq & \|S\| \\
				& = & (2r_m-1)(q_m +p_m)c_m + (2r_m-1)q_m n_m + (r_m-1)p_m n_m +r_m p_m n_m \\
			& =	& (2r_m-1)(q_m +p_m)(c_m + n_m)
			\sim \frac{16}{3 \sqrt 3}m^3.		
\end{eqnarray*} 
This proves the desired lower bound.

As for the upper bound, 
it follows from the same approach as in Section \ref{useofsteinitz}, that is on Steinitz lemma.
In the present case we use the result of \cite{Banasz87} (see Remark 3), which gives that the Steinitz constant of the unit ball $B'_1$ 
of $\R^3$ is bounded from above by $d+1/d -1=7/3$ when $d=3$. 
This gives
$$
\DD (\Bcal^{(3)}_m) \lesssim \vol\ \big( B'_{7m/3} \big) = \frac43 \pi \left( \frac{7m}{3} \right)^3 = \frac{4 \cdot 7^3}{3^4} \pi m^3 = \frac{1372}{81} \pi m^3.
$$

\section{Optimality in dimension 3: Theorem \ref{thmd4}}
\label{optindim33}

The proof goes in an analogous way as in dimension 2, see Section \ref{prtheopt}.
However, things are less ``visible" here. Therefore, a more formal proof is definitely needed. 
The good news is that it will be adapted to a later generalisation. 
The following lemma is the central part of the proof.

\begin{lemma} 
\label{wC}
Let $w_1,w_2,w_3,w_4$ be positively dependent vectors in $\R^3$ and let 
\begin{eqnarray*}
D(w_1,w_2,w_3,w_4) & = & \Conv ( w_1,w_2, w_3, w_4,\,
w_1+w_2, w_1+w_3, w_1+w_4, w_2+w_3, w_2+w_4, w_3+w_4, \\
&& \hspace{2.9cm} w_1+w_2+w_3, w_1+w_2+w_4, w_1+w_3+w_4, w_2+w_3+w_4 ).
\end{eqnarray*}
Set the four parallelepipeds 
\begin{eqnarray*}
&D_1=\{\alpha w_2 + \beta w_3+\gamma w_4 : 0 \le \alpha,\beta,\gamma \le 1\}, & \\
&D_2=\{\alpha w_1 + \beta w_3 +\gamma w_4 : 0 \le \alpha,\beta, \gamma \le 1\}, & \\
&D_3=\{\alpha w_1 + \beta w_2+\gamma w_4 : 0 \le \alpha,\beta, \gamma \le 1\}, & \\
&D_4 =\{\alpha w_1 + \beta w_2+\gamma w_3 : 0 \le \alpha,\beta, \gamma \le 1\}.
\end{eqnarray*} 
Then, 
$$
D (w_1,w_2,w_3,w_4)=\displaystyle\bigcup^4_{i=1} D_i = \displaystyle\bigcup^4_{i=1} (w_i+D_i).
$$
\end{lemma}

Notice that the polyhedron $D(w_1,w_2,w_3,w_4)$ appearing here is a rhombic dodecahedron \cite{Coxeter} (well, an irregular one in general), as illustrated by Figure 3.

\tdplotsetmaincoords{60}{85}

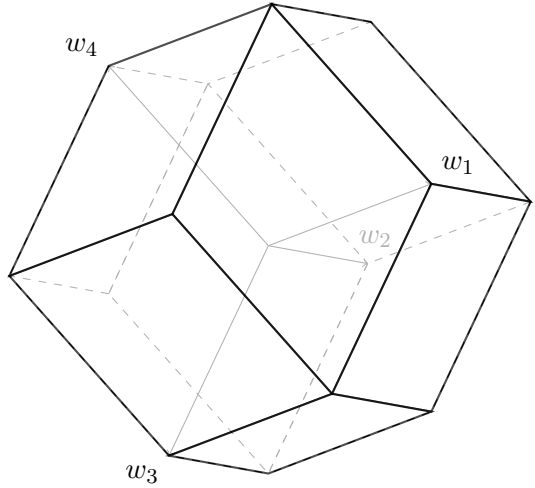
\begin{figure}[h!]
\begin{center}
\begin{tikzpicture}[tdplot_main_coords,scale=0.35]

	\coordinate (0) at (0,0,0);
	

	\coordinate (20) at (4,4,4);
	\coordinate (20bis) at (4.5,3.75,4);
	\coordinate (21) at (-4,3,-3);
	
	\coordinate (23) at (4,-3,-4);
	\coordinate (26) at (-3,-4,4);


	\coordinate (17) at (0,7,1);
	
	\coordinate (14) at (8,1,0);
	
	\coordinate (13) at (1,0,8);
	
	\coordinate (15) at (0,0,-7);
	
	\coordinate (16) at (-7,-1,1);
	
	\coordinate (18) at (1,-7,0);
	
	
	\coordinate (19) at (4,4,-3);
	
	\coordinate (22) at (-3,3,5);
	
	\coordinate (24) at (5,-3,4);
	
	\coordinate (25) at (-3,-4,-3);
	
 
 \draw [->] (0) -- (20);
\draw [->] (0) -- (21);
\draw [->] (0) -- (23);
\draw [->] (0) -- (26);
\draw (-3.5,2.1,-2.5) node [anchor=south west] {$w_2$};
 

	\draw[dashed] (25) -- (18) -- (23) -- (15);	
         \draw[dashed] (16) -- (25) -- (15) -- (21);
	\draw[dashed] (17) -- (22) -- (16) -- (21);
	\draw[dashed] (25) -- (16) -- (26) -- (18);
	\draw[dashed] (22) -- (13);
	\draw[dashed] (17) -- (21) -- (15) -- (19) -- cycle;





        \draw[thick,fill=black!0,opacity=0.7] (17) -- (20) -- (13) -- (22) -- cycle;
        \draw[thick,fill=black!0,opacity=0.7] (20) -- (14) -- (24) -- (13) -- cycle;
        \draw[thick,fill=black!0,opacity=0.7] (14) -- (23) -- (18) -- (24) -- cycle;
        \draw[thick,fill=black!0,opacity=0.7] (17) -- (20) -- (14) -- (19) -- cycle;
        \draw[thick,fill=black!0,opacity=0.7] (19) -- (14) -- (23) -- (15) -- cycle;
        \draw[thick,fill=black!0,opacity=0.7] (13) -- (24) -- (18) -- (26) -- cycle;
	
\draw (20bis) node [anchor=south west] {$w_1$};
\draw (23) node [anchor=north east] {$w_3$};
\draw (26) node [anchor=south east] {$w_4$};

\end{tikzpicture}
\end{center}
\caption{A representation of $D(w_1,w_2,w_3,w_4)$ as a translucent polyhedron.}
\end{figure}

\begin{proof}
We define the four positive cones:
\begin{eqnarray*}
& K_1 = \{ \alpha w_2 + \beta w_3 + \gamma w_4 : \alpha, \beta, \gamma \ge 0\}, & \\
& K_2 = \{ \alpha w_1 + \beta w_3 + \gamma w_4 : \alpha, \beta, \gamma \ge 0\}, & \\
& K_3 = \{ \alpha w_1 + \beta w_2 + \gamma w_4 : \alpha, \beta, \gamma \ge 0\}, & \\
& K_4 = \{ \alpha w_1 + \beta w_2 + \gamma w_3 : \alpha, \beta, \gamma \ge 0\}.
\end{eqnarray*}
We first show that 
\begin{equation}
\label{cones}
\displaystyle\bigcup^4_{i=1} K_i = \R^3.
\end{equation}

By positive dependency, there are positive $x,y,z$ such that 
$$
w_4 = -  (xw_1+yw_2+zw_3)
$$
and, as explained at the beginning of Section \ref{prtheopt}, $(w_1, w_2, w_3$) has to be a basis of $\R^3$. 
Let $t \in \R^3$, say $t= \alpha w_1+ \beta w_2+ \gamma w_3$ for some real numbers 
$\alpha, \beta,\gamma$. We assume that $t \not\in \displaystyle\bigcup^3_{i=1} K_i$.
We compute that
$$
t = \frac{ \beta x - \alpha y}{x}  w_2 + \frac{ \gamma x- \alpha z}{x}  w_3 - \frac{\alpha}{x} w_4
$$
so that $t \not\in K_1$ implies, since $x >0$,
\begin{equation}
\label{ineq11}
\beta x - \alpha y  <0,\quad {\rm or}\quad \gamma x - \alpha z <0 ,\quad {\rm or}\quad \alpha >0.
\end{equation}
In the same way, $t \not\in K_2$ implies
\begin{equation}
\label{ineq22}
\alpha y - \beta x  <0,\quad {\rm or}\quad \gamma y- \beta z  <0 ,\quad {\rm or}\quad \beta > 0
\end{equation}
and  $t \not\in K_3$ implies
\begin{equation}
\label{ineq33}
\alpha z - \gamma x  <0,\quad {\rm or}\quad \beta z - \gamma y <0 ,\quad {\rm or}\quad  \gamma >0.
\end{equation}

We have different cases to consider: 

(i) Assume first that  $\beta x - \alpha y  <0$ and $\gamma x - \alpha z <0$ so that \eqref{ineq11} is satisfied.
By \eqref{ineq22}, one has
$$
\gamma y - \beta z <0 ,\quad {\rm or}\quad \beta > 0
 $$
and by \eqref{ineq33}, one has
$$
\beta z - \gamma y <0 ,\quad {\rm or}\quad  \gamma >0.
$$
Thus if $\gamma y - \beta z <0$, one must have $\gamma >0$. But in this case
$z\beta > \gamma y >0$ and thus $\beta >0$ by the positivity of $y$ and $z$.
 In the other case, $\gamma y - \beta z > 0$, one must have $\beta >0$ and then
$\gamma > \beta z/y>0$ as well. In the case where $\gamma y - \beta z =0$, then 
directly $\beta, \gamma > 0$. Thus in any case $\beta, \gamma > 0$ and using our starting 
assumption that $\beta x - \alpha y  <0$  and the positivity of $x$ and $y$, one gets $\alpha>0$ 
so that finally $t \in K_4$.

(ii) Assume now that  $\beta x - \alpha y  <0$ and $\gamma x - \alpha z > 0$. 
Equation \eqref{ineq11} is still satisfied. By \eqref{ineq22}, one has
$$
\gamma y - \beta z <0 ,\quad {\rm or}\quad \beta > 0.
$$
If $\beta >0$ then since $\alpha y > \beta x$, one has by positivity of $x$ and $y$, $\alpha>0$. 
And since $\gamma x - \alpha z > 0$, one has $\gamma x >  \alpha z$ thus by positivity of $x$ and $z$, $\gamma>0$, 
so that finally $t \in K_4$.

If $\gamma y - \beta z <0$, then we have the three inequalities
$$
\gamma y < \beta z,\quad \beta x <\alpha y ,\quad \alpha z < \gamma x.
$$
But then, due to the positivity of $x,y$ and $z$, 
$$
\gamma < \frac{z}{y} \beta < \frac{zy}{yx}  \alpha =  \frac{z}{x}  \alpha < \gamma,
$$
a contradiction.

(iii) The case where $\beta x - \alpha y  >0$ and $\gamma x - \alpha z < 0$ is symmetrical and left to the reader.

(iv) Finally if $\beta x - \alpha y  >0$ and $\gamma x - \alpha z >0$, one must have by  \eqref{ineq11} $\alpha >0$.
But then, as above, $\beta x - \alpha y  >0$ implies $\beta >0$ and $\gamma x - \alpha z >0$ implies $\gamma>0$. 
Finally, we obtain that $t \in K_4$ as well.

(v) It remains to check the cases where one of the inequalities is an equality, say $\beta x - \alpha y =0$ 
(the other cases are symmetrical).
But in this case  \eqref{ineq11} implies: $\gamma x - \alpha z <0$ or $\alpha >0$. 

If $\gamma x - \alpha z <0$, then by \eqref{ineq33}, either $\beta z - \gamma y <0$ in which case 
$\beta >0$ by \eqref{ineq22} and then so do $\alpha = \beta x/y$ and $\gamma > \beta z/y$, by the positivity of $x,y$ and $z$; 
or $\gamma > 0$ but then so does $\alpha >\gamma x /z$ and then $\beta = \alpha y/x$. In both cases, $t \in K_4$.

And if $\gamma x - \alpha z \geq 0$ then one must have $\alpha >0$, which in turn gives, with $\beta x - \alpha y =0$, that $\beta >0$.
And since we are in the case $\gamma x \geq  \alpha z $ we get $\gamma \geq 0$.

In short, one has $t \in K_4$ in all five cases, and \eqref{cones} is proved.
\medskip

Now we prove that, for $i=1,2,3$ and 4, 
\begin{equation}
\label{DK}
D_i = K_i \cap D(w_1,w_2,w_3,w_4).
\end{equation}
Indeed $D_i$ is clearly contained in $K_i$ and contained in $D(w_1,w_2,w_3,w_4)$ in view of the fact 
that an element of the form 
$\alpha w_j+ \beta w_k + \gamma w_l$ with $\{j,k,l\} = \{1,2,3,4 \} \setminus \{i\}$ 
and, say (up to a reordering) $0 \leq \alpha \leq \beta \leq \gamma \leq 1$, 
can be rewritten as a combination with positive coefficients summing to 1 as 
$$
\alpha ( w_j+ w_k + w_l) + (\beta - \alpha) (w_k + w_l)+ (\gamma-\beta) w_l + (1- \gamma)\cdot 0
$$
and is therefore a barycenter of $w_j+ w_k + w_l$, $w_k + w_l$, $w_l$ and $0$. 
Since $w_1,w_2,w_3,w_4$ are positively dependent, $0$ itself belongs to $D(w_1,w_2,w_3,w_4)$ and thus, so does such a barycenter. 
This proves 
$$
D_i \subseteq K_i \cap D(w_1,w_2,w_3,w_4).
$$

Conversely, any element $t$ of $K_i$ can, by definition, be written as 
\begin{equation}
\label{ttt}
t=\alpha w_j+ \beta w_k + \gamma w_l
\end{equation}
with $\{j,k,l\} = \{1,2,3,4 \} \setminus \{i\}$ 
and $\alpha, \beta, \gamma \geq 0$. If $t$ is also an element of $D(w_1,w_2,w_3,w_4)$ 
it can be written as a barycenter of
\begin{eqnarray*}
&w_1,\ w_2,\ w_3,\ w_4,\quad w_1+w_2,\ w_1+w_3,\ w_1+w_4,\ w_2+w_3,\ w_2+w_4,\ w_3+w_4, &\\
&w_1+w_2+w_3,\quad w_1+w_2+w_4,\quad w_1+w_3+w_4,\quad w_2+w_3+w_4 &
\end{eqnarray*}
therefore it can be decomposed as a sum of the form 
\begin{equation}
\label{sumde4}
t = \alpha' w_j+ \beta' w_k + \gamma' w_l + \delta w_i
\end{equation}
with coefficients satisfying $0 \leq \alpha' , \beta' , \gamma' , \delta \leq 1$. But, by positive dependency,
$$
w_i = - \big( x_{i,j} w_j+ x_{i,k} w_k + x_{i,l} w_l)
$$
with $x_{i,j}, x_{i,k}, x_{i,l} >0$. Therefore, if we eliminate $w_i$ in \eqref{sumde4}, we get
\begin{eqnarray*}
t & = & \alpha' w_j+ \beta' w_k + \gamma' w_l - \delta  \big(  x_{i,j} w_j+ x_{i,k} w_k + x_{i,l} w_l  \big) \\
  & = & ( \alpha' - \delta  x_{i,j} ) w_j+ (\beta' - \delta x_{i,k})  w_k + (\gamma'  - \delta  x_{i,l}) w_l.
\end{eqnarray*}
Since the decomposition \eqref{ttt} is unique (on recalling that $(w_j, w_k, w_l)$ is a basis), one has
$$
\alpha = \alpha' - \delta  x_{i,j}\leq \alpha' \leq 1 ,\quad   \beta = \beta' - \delta x_{i,k} \leq \beta' \leq 1 \quad 
\gamma =\gamma'  - \delta  x_{i,l} \leq \gamma' \leq 1,
$$
and these three coefficients being also non-negative by assumption, one finally has $t \in D_i$. And 
\eqref{DK} is proved.

\tdplotsetmaincoords{60}{85}

\begin{figure}[h!]
\begin{center}
\begin{tikzpicture}[tdplot_main_coords,scale=0.35]

	\coordinate (0) at (0,0,0);
	

	\coordinate (20) at (4,4,4);
	\coordinate (20bis) at (4.5,3.75,4);
	\coordinate (21) at (-4,3,-3);
	
	\coordinate (23) at (4,-3,-4);
	\coordinate (26) at (-3,-4,4);


	\coordinate (17) at (0,7,1);
	
	\coordinate (14) at (8,1,0);
	
	\coordinate (13) at (1,0,8);
	
	\coordinate (15) at (0,0,-7);
	
	\coordinate (16) at (-7,-1,1);
	
	\coordinate (18) at (1,-7,0);
	
	
	\coordinate (19) at (4,4,-3);
	
	\coordinate (22) at (-3,3,5);
	
	\coordinate (24) at (5,-3,4);
	
	\coordinate (25) at (-3,-4,-3);
 
 \draw [->,dashed] (0) -- (20);
\draw [->,dashed] (0) -- (21);
\draw [->,dashed] (0) -- (23);
\draw [->,dashed] (0) -- (26);
\draw (-3.5,2.1,-2.5) node [anchor=south west] {$w_2$};
 
	

	\draw[thick] (15) -- (25) -- (18) -- (23) -- cycle;	
	\draw[thick] (21) -- (16) -- (25) -- (15) -- cycle;
	\draw[thick] (17) -- (22) -- (16) -- (21) -- cycle;
	\draw[thick] (16) -- (26) -- (18) -- (25) -- cycle;
	\draw[thick] (22) -- (13) -- (26) -- (16) -- cycle;
	\draw[thick] (17) -- (21) -- (15) -- (19) -- cycle;





        \draw[thick,fill=black!15,opacity=1] (17) -- (20) -- (13) -- (22) -- cycle;
        \draw[thick,fill=black!0,opacity=1] (20) -- (14) -- (24) -- (13) -- cycle;
        \draw[thick,fill=black!0,opacity=1] (14) -- (23) -- (18) -- (24) -- cycle;
        \draw[thick,fill=black!30,opacity=1] (17) -- (20) -- (14) -- (19) -- cycle;
        \draw[thick,fill=black!30,opacity=1] (19) -- (14) -- (23) -- (15) -- cycle;
        \draw[thick,fill=black!0,opacity=1] (13) -- (24) -- (18) -- (26) -- cycle;
	
\draw (20bis) node [anchor=south west] {$w_1$};
\draw (23) node [anchor=north east] {$w_3$};
\draw (26) node [anchor=south east] {$w_4$};

\begin{scope}[shift={(2,18)}]
	
	\coordinate (0) at (0,0,0);
	

	\coordinate (20) at (4,4,4);
	\coordinate (20bis) at (4.5,3.75,4);
	\coordinate (21) at (-4,3,-3);
	
	\coordinate (23) at (4,-3,-4);
	\coordinate (26) at (-3,-4,4);
	

	\coordinate (17) at (0,7,1);
	
	\coordinate (14) at (8,1,0);
	
	\coordinate (13) at (1,0,8);
	
	\coordinate (15) at (0,0,-7);
	
	\coordinate (16) at (-7,-1,1);
	
	\coordinate (18) at (1,-7,0);
	
	
	\coordinate (19) at (4,4,-3);
	
	\coordinate (22) at (-3,3,5);
	
	\coordinate (24) at (5,-3,4);
	
	\coordinate (25) at (-3,-4,-3);
 
 \draw [->,dashed] (0) -- (20);
\draw [->,dashed] (0) -- (21);
\draw [->,dashed] (0) -- (23);
\draw [->,dashed] (0) -- (26);
\draw (-3.5,2.1,-2.5) node [anchor=south west] {$w_2$};

 

	\draw[thick] (15) -- (25) -- (18) -- (23) -- cycle;	
	\draw[thick] (21) -- (16) -- (25) -- (15) -- cycle;
	\draw[thick] (17) -- (22) -- (16) -- (21) -- cycle;
	\draw[thick] (16) -- (26) -- (18) -- (25) -- cycle;
	\draw[thick] (22) -- (13) -- (26) -- (16) -- cycle;
	\draw[thick] (17) -- (21) -- (15) -- (19) -- cycle;





        \draw[thick,fill=black!60,opacity=1] (17) -- (20) -- (13) -- (22) -- cycle;
         \draw[thick,fill=black!60,opacity=1] (20) -- (14) -- (24) -- (13) -- cycle;
        \draw[thick,fill=black!15,opacity=1] (14) -- (23) -- (18) -- (24) -- cycle;
        \draw[thick,fill=black!60,opacity=1] (17) -- (20) -- (14) -- (19) -- cycle;
        \draw[thick,fill=black!15,opacity=1] (19) -- (14) -- (23) -- (15) -- cycle;
        \draw[thick,fill=black!30,opacity=1] (13) -- (24) -- (18) -- (26) -- cycle;

\draw (20bis) node [anchor=south west] {$w_1$};
\draw (23) node [anchor=north east] {$w_3$};
\draw (26) node [anchor=south east] {$w_4$};

\end{scope}	
\end{tikzpicture}
\end{center}

\caption{Each of the two figures represents $D(w_1,w_2,w_3,w_4)$ as an opaque polyhedron. On the left, the three visible faces of $D_2$ are in white, the only visible face of $D_3$ in light grey and the two visible faces of $D_4$ in medium grey ($D_1$ is entirely hidden at the back of the solid). On the right, the three visible faces of $w_1+D_1$ are in dark grey, the only visible face of $w_4+D_4$ in medium grey and the two visible faces of $w_3+D_3$ in light grey ($w_2+D_2$ is entirely hidden at the back of the solid).}
\end{figure}
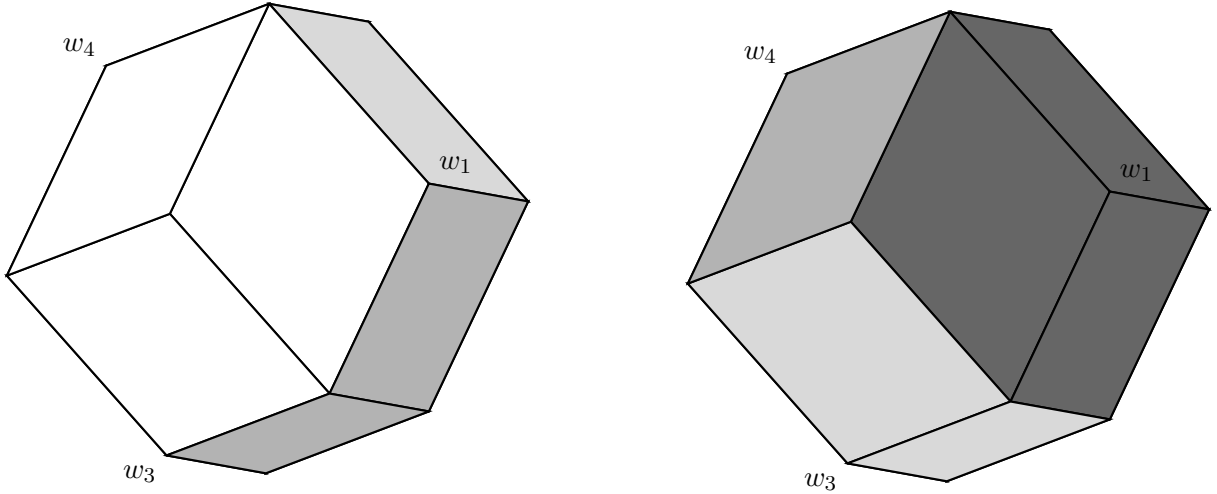

We are now ready to prove our lemma (see Figure 4). Indeed, by \eqref{cones} and \eqref{DK}, one has
\begin{eqnarray*}
D(w_1,w_2,w_3,w_4) & = & D(w_1,w_2,w_3,w_4) \bigcap \big( \bigcup^4_{i=1} K_i \big) \\
				 & = & \bigcup^4_{i=1} \big( K_i \cap D(w_1,w_2,w_3,w_4) \big)
				  = \bigcup^4_{i=1} D_i,
\end{eqnarray*}
and the first equality of the lemma is proved.

To show the second equality, we apply the first identity to the set of vectors $-w_1,-w_2,-w_3, -w_4$ which are positively dependent as well. 
We obtain
$$
D (- w_1, -w_2, -w_3, -w_4)=\displaystyle\bigcup^4_{i=1} D'_i
$$
where 
\begin{eqnarray*}
D'_1=\{-\alpha w_2 - \beta w_3 - \gamma w_4 : 0 \le \alpha,\beta, \gamma \le 1\}, &&
D'_2=\{- \alpha w_1 - \beta w_3 - \gamma w_4 : 0 \le \alpha,\beta, \gamma \le 1\}, \\
D'_3=\{- \alpha w_1 - \beta w_2- \gamma w_4 : 0 \le \alpha,\beta, \gamma \le 1\} & {\rm and} &
D'_4=\{- \alpha w_1 - \beta w_2- \gamma w_3 : 0 \le \alpha,\beta, \gamma \le 1\}.
\end{eqnarray*} 
We now translate this identity by the vector $w_1+w_2+w_3+w_4$. 
This gives
\begin{equation}
\label{wD}
(w_1+w_2+w_3+w_4) + D (- w_1, -w_2, -w_3, -w_4)=\displaystyle\bigcup^4_{i=1} (w_1+w_2+w_3 +w_4+ D'_i ).
\end{equation}
It turns out that the left-hand side is equal to $D(w_1, w_2, w_3,w_4)$, 
while the $(w_1+w_2+w_3 +w_4+D'_i )$'s can be computed to be $w_i + D_i$.
Thus \eqref{wD} gives
$$
D(w_1, w_2, w_3,w_4) = \displaystyle\bigcup^4_{i=1} (w_i+D_i ).
$$

\end{proof}

The following corollary is, {\em mutatis mutandis}, what we have done in the two-dimensional case.

\begin{corollary}
\label{coroend3}
Let $n \ge 4$ and $u_1,\dots, u_n$ be a minimal zero-sum sequence  in $\mathbb{Z}^3$ 
having a support of size $4$, say $\{w_1,w_2,w_3,w_4\}$.
Then, $n$ is at most the number of integral points in $D(w_1,w_2,w_3,w_4)$.
\end{corollary}

\begin{proof}

Since $\sum^n_{i=1} u_i=0$, the vectors $w_1,w_2,w_3,w_4$ are positively dependent. 
We shall prove by induction that there is a permutation  $\sigma$ of $\{1,\dots, n\}$ such that 
$s_j=\sum^j_{i=1} u_{\sigma(i)} \in D(w_1,w_2,w_3,w_4)$ for every $j \in \{1,\dots,n\}$. 

We choose $\sigma (1)=1$. Suppose now that for some positive integer $k<n $, the indices $\sigma (1),\dots, \sigma (k)$ are already chosen 
so that $s_j \in D(w_1,w_2,w_3,w_4)$ for all $j \in \{1,\dots, k\}$. 
By Lemma \ref{wC}, there is a value $i_0$ in $\{1,2,3,4\}$ such that $s_k \in D_{i_0}$. 
We choose $\sigma (k+1)\in \{1,\dots, n\} \setminus \{ \sigma (1),\dots, \sigma (k) \}$ such that $u_{\sigma (k+1)}=w_{i_0}$. 

Notice that there must be such a value for $\sigma (k+1)$ otherwise all the remaining values in the sequence $u_1,\dots, u_n$ 
(those with index in $\{1,\dots, n\} \setminus \{ \sigma (1),\dots, \sigma (k) \}$) 
are equal to one of the three other values $w_j, w_{j'}$ or $w_{j''}$ (with $\{j,j',j''\} = \{1,2,3,4\} \setminus \{i_0\}$). 
Thus 
$$
0 = \sum^n_{i=1} u_i = 
\sum^k_{i=1} u_{\sigma (i)} + \sum_{i \in \{1,\dots, n\} \setminus \{ \sigma (1),\dots, \sigma (k) \}} u_i \quad \in \quad 
D_{i_0} + \big ( (\NNb \cdot w_j + \NNb \cdot w_{j'}+ \NNb \cdot w_{j''}) \setminus \{ 0 \} \big),
$$
which is a contradiction.

Then, by Lemma \ref{wC},
$$
s_{k+1} = s_k +u_{\sigma (k+1)}=s_k+w_{i_0} \in w_{i_0}+D_{i_0} \subseteq \bigcup^4_{i=1} (w_i+D_i) = D(w_1,w_2,w_3,w_4),
$$
and the induction step is completed.

Since our sequence has no non-trivial zero-sum subsequence, it follows that $s_1,\dots,s_n$ are distinct elements 
of $D(w_1,w_2,w_3,w_4)$, which gives the desired result.
\end{proof}

We now proceed with the very proof of Theorem \ref{thmd4}. 
As in Section \ref{prtheopt}, the case where the support of our sequence consists of four coplanar vectors leads back to 
the two-dimensional case where we already derived a bound being only quadratic in $m$. 

We now investigate the case where the support (of size 4) of our minimal zero-sum sequence spans $\R^3$. 
In view of the fact that the sequence is zero-sum, the four vectors of the support of our sequence must be positively dependent.
By Corollary \ref{coroend3}, the value of $\DD^{(4)} (\Bcal^{(3)}_m)$ is bounded from above by the maximal number of integral points in 
a polyhedron of the form $D(w_1,w_2,w_3, w_4)$ with the $w_i$'s in $\Bcal^{(3)}_m$. 
Approximating this number of integral points by its volume is valid, in view of what we called 
the Principle of approximation of the volume by the number of integral points (see Section \ref{useofsteinitz}).

It remains to show that the volume of the polyhedron $D(w_1,w_2,w_3,w_4)$
is maximised for $w_1,w_2,w_3, w_4$ forming a regular tetrahedron, in which case the volume is $(16/3 \sqrt 3) m^3$. 

This last fact may seem classical or simple. However, we have not been able to find a reference for it.
Up to a homothety and a dilation, we can assume that $w_1,w_2,w_3, w_4$ lie on the unit sphere.
Up to a change of axis and a rotation, we may choose $w_1$ and $w_2$ of the form
$$
w_1 = 
\begin{pmatrix}
1 \\
0 \\
0
\end{pmatrix}
\quad
{\rm and}\quad 
w_2 = 
\begin{pmatrix}
\cos \vartheta \\
\sin \vartheta \\
0
\end{pmatrix}
$$
with $0 \leq \vartheta \leq \pi$.
As for the two other points, we may look for them as
$$
w_3 = 
\begin{pmatrix}
\cos \vartheta_1 \cos \varphi_1 \\
\cos \vartheta_1 \sin \varphi_1 \\
\sin \vartheta_1
\end{pmatrix}
\quad
{\rm and}\quad 
w_4 = 
\begin{pmatrix}
\cos \vartheta_2 \cos \varphi_2  \\
\cos \vartheta_2 \sin \varphi_2\\
\sin \vartheta_2
\end{pmatrix}
$$
with $\vartheta_1 \in ] 0, \pi/2 ], \vartheta_2 \in [- \pi/2, 0 [$ and $\pi  \leq \varphi_1, \varphi_2 \leq 2 \pi$.

The volume of the rhombic dodecahedron $D(w_1,w_2,w_3,w_4) $ is easily seen to be
\begin{eqnarray*}
V & = & \sin \vartheta \big( \sin \vartheta_1-  \sin \vartheta_2 \big) +
| \cos \vartheta_1  \sin \varphi_1 \sin \vartheta_2 -  \sin \vartheta_1 \cos \vartheta_2 \sin \varphi_2 | \\
&& + \big| \cos \vartheta \big( \cos \vartheta_1  \sin \varphi_1 \sin \vartheta_2 -  \sin \vartheta_1 \cos \vartheta_2 \sin \varphi_2 \big)
- \sin \vartheta  \big( \cos \vartheta_1  \cos \varphi_1 \sin \vartheta_2 -  \sin \vartheta_1 \cos \vartheta_2 \cos \varphi_2 \big)  \big|.
\end{eqnarray*}

Computing the maximum of this five-variable function is an unpleasant task. But it is classically done using partial derivative calculus.
One check first that
$\varphi_1 = \varphi_2$, and then that $\vartheta_1 = - \vartheta_2$ and $\varphi_1 = \frac{\vartheta}{2} + \pi$.
Finally, we obtain $V$ as a quite simple function of only $\vartheta$ and  $\vartheta_1$:
$$
V (\vartheta, \vartheta_1) = 2 \sin \vartheta \sin \vartheta_1 + 2 \sin (\vartheta /2) \sin 2\vartheta_1 .
$$
Differentiating with respect to $\vartheta$ gives
\begin{equation}
\label{theta1}
\cos \vartheta_1 = - \frac{\cos \vartheta}{\cos (\vartheta /2)}.
\end{equation}
Differentiating finally with respect to $\vartheta_1$ gives
$$
0 = \frac12 \frac{\partial  V}{\partial \vartheta_1} = \sin \vartheta \cos \vartheta_1 + 2 \sin (\vartheta /2) \cos 2\vartheta_1 .
$$
Replacing the occurrences of $\vartheta_1$ using \eqref{theta1} leads to
$$
(1+ \cos \vartheta ) \cos^2 (\vartheta /2) =2  \cos^2 \vartheta
$$
that is $X(2X-1)-2 (2X-1)^2+X=0$ if $X =  \cos^2 (\vartheta /2)$, which gives $X= 1/3$ and finally $\cos \vartheta = -1/3$. 
All the other parameters can then be deduced so that finally we obtain
$$
w_1 = 
\begin{pmatrix}
1 \\
0 \\
0
\end{pmatrix}
\quad
w_2 = 
\begin{pmatrix}
-1/3 \\
2 \sqrt{2}/3\\
0
\end{pmatrix}
\quad
w_3 = 
\begin{pmatrix}
-1/3\\
- \sqrt 2 / 3\\
\sqrt 2 / \sqrt 3
\end{pmatrix}
\quad
{\rm and}\quad 
w_4 = 
\begin{pmatrix}
-1/3\\
- \sqrt 2 / 3\\
- \sqrt 2 / \sqrt 3
\end{pmatrix}
$$
that is four vertices of the regular tetrahedron for $(w_1,w_2,w_3,w_4)$, as announced.

\section{Investigating sequences with support of size $d+1$ in dimension $d$: Conjecture \ref{conj78}}
\label{sec88}

In this section, we would like to give some support to Conjecture \ref{conj78} and explain that an important part of it is more than conjectural.
In dimension $d=2$ or 3, we already observed that the longest minimal zero-sum sequence having a support of size $d+1$ is obtained for a support the elements of which are (almost) on the sphere of radius $m$ and (almost) equidistant from one another. 
Then, the length of such a minimal zero-sum sequence is asymptotic to the area of the regular hexagon of side $m$ in dimension 2 and to the volume of the regular rhombic dodecahedron of side $m$ in dimension 3.

Our conjecture is that the same phenomenon should happen in higher dimensions. 
Indeed, the two lemmas of the previous section can be adapted to higher dimensions. 

\begin{lemma} 
\label{wCCC}
Let $d \geq 4$ be an integer and $w_1,\dots , w_{d+1}$ be positively dependent vectors in $\R^d$and let $L(w_1,\dots , w_{d+1})$ 
be the convex hull of all $(2^{d+1} -2)$ sums of 1 to $d$ vectors among the $w_i$'s, namely
\begin{eqnarray*}
 L(w_1,\dots , w_{d+1})& = & \Conv ( w_1,\dots, w_{d+1}, \, \dots,
w_1+w_2, w_1+w_3, \dots, w_d+w_{d+1}, \dots \\
&& \hspace{2.5cm}  \dots, w_1+ w_2 +\cdots + w_{d-1} +w_{d+1},\ w_1+ w_2 +\cdots +w_{d-1} +w_{d}).
\end{eqnarray*}
Set for $i=1,\dots, d+1$ the $d+1$ hyperparallelepipeds 
\begin{eqnarray*}
L_{i} &=  & \{\alpha_1 w_1 + \alpha_2 w_2 + \cdots + \alpha_{i-1} w_{i-1} + \alpha_{i+1} w_{i+1}+ \cdots +\alpha_{d+1} w_{d+1}\ :\\
	& & \hspace{4cm} 0 \le \alpha_1, \alpha_2, \dots , \alpha_{i-1}, \alpha_{i+1}, \dots, \alpha_{d+1}\le 1\}
\end{eqnarray*}

Then, 
$$
L(w_1,\dots , w_{d+1})=\displaystyle\bigcup^{d+1}_{i=1} L_i = \displaystyle\bigcup^{d+1}_{i=1} (w_i+L_i).
$$
\end{lemma}

From this result, one can deduce, as in lower dimensions, an algorithmic lemma.

\begin{corollary}
Let $d \geq 4$ be an integer, $n \ge d+1$ and $u_1,\dots, u_n$ be a minimal zero-sum sequence  in $\mathbb{Z}^{d}$ 
having a support of size $d+1$, say $\{w_1,w_2,\dots , w_{d+1}\}$.
Then, $n$ is at most the number of points with integer coordinates in $L(w_1,\dots , w_{d+1})$.
\end{corollary}

Having this at hand, and approximating the number of integral points of $L(w_1,\dots , w_{d+1})$ using the Principle 
of approximation of the volume by the number of integral points, we are led to compute the volume of such an 
$L(w_1,\dots , w_{d+1})$ where the $w_i$'s are the elements of the support of the sequence, 
and then maximise it. 
We conjecture that the $d+1$ vectors $w_i$'s leading to the maximum volume
will be obtained as in lower dimensions that is, as the vertices of a regular polytope on $d+1$ points regularly spaced 
on the unit-sphere of $\RR^d$ (up to a homothety), that is with a given fixed distance from one another. 
This would generalise in higher dimensions the constructions leading to the equilateral triangle in dimension 2 and 
the regular tetrahedron in dimension 3.
For instance, in dimension $4$, we would obtain a regular pentachoron.

Then, the idea is that the longest minimal zero-sum sequence with support of size $d+1$ will 
be obtained by perturbing slightly these $d+1$ points so that they become integral  and 
then move them again so that they have the coprimality conditions which ensure that our 
sequence is a minimal zero-sum sequence. 

That said, in order to justify numerically our conjecture, we have to find how to place 
$d+1$ points evenly spaced on the unit-sphere of $\RR^d$ 
and compute the volume of the polytope generated by these points.

Define
$
\sigma_1^{(1)} = (-1)\quad {\rm and}\quad \sigma_2^{(1)} = (1).
$
They are one-dimensional points on the unit-sphere $\{-1, 1 \}$ of $\RR$ being ``optimally" far from each other. 
The distance between these two points is 2. The length of the segment they define, namely $[-1,1]$, is $V_1=2$.

Define
$$
\sigma_1^{(2)} =
\begin{pmatrix}
-1/2 \\
-\sqrt{3}/2
\end{pmatrix},\quad
\sigma_2^{(2)} =
\begin{pmatrix}
-1/2 \\
\sqrt{3}/2
\end{pmatrix},\quad
\sigma_3^{(2)} =
\begin{pmatrix}
1 \\
0
\end{pmatrix}.
$$
These are two-dimensional points on the unit-sphere (the unit circle) of $\RR^2$ that are ``optimally" far from each other and regularly spaced. 
The distance between any two of these three points is $\sqrt{3}$.
These are the vertices of an equilateral triangle. The area of the triangle defined by these three points is $3 \sqrt{3}/4$.

In dimension 3, we obtain the four vectors: 
$$
\sigma_1^{(3)} =
\begin{pmatrix}
-1/3 \\
-\sqrt{2}/3\\
-\sqrt{2}/\sqrt{3}
\end{pmatrix},\quad
\sigma_2^{(3)} =
\begin{pmatrix}
-1/3 \\
-\sqrt{2}/3\\
\sqrt{2}/\sqrt{3}
\end{pmatrix},\quad
\sigma_3^{(3)} =
\begin{pmatrix}
-1/3\\
2 \sqrt{2}/3\\
0
\end{pmatrix},\quad
\sigma_4^{(3)} =
\begin{pmatrix}
1\\
0\\
0
\end{pmatrix}.
$$
These are the vertices of a regular tetrahedron. The distance between any two of these four points is $2 \sqrt{2}/\sqrt{3}$. 
Up to a rotation, these two last sets of points coincide with the supports of the sequences we used above, 
in our two- and three-dimensional constructions.

In general, we can always find $d+1$ points regularly spaced on the unit-sphere of $\RR^d$.
To construct such distributions of points, we may proceed by induction. 
Assuming that $(\sigma_i^{(d-1)})_{1 \leq i \leq d}$ is known, we define a finite sequence of $d$-dimensional vectors 
as follows: 
$$
\sigma_i^{(d)} = 
\begin{pmatrix}
-1/d\\
\sqrt{1 - 1/d^2}\ \sigma_i^{(d-1)} 
\end{pmatrix}
\ \text{ if } \ i \leq d, \ \text{ and } \ \sigma_i^{(d)} = \begin{pmatrix}
1\\
0_d
\end{pmatrix} \ \text{ if } \ i =d+1,$$
where $0_d$ is the all-zero vector of dimension $d$.

The elements of such a sequence are easily shown to be on the unit-sphere and to be pairwise at distance 
$$
\delta_d =  \sqrt{ 2 \ \frac{d+1}{d}},
$$
a quantity satisfying the immediate induction: $\delta_d = \sqrt{1-1/d^2}\ \delta_{d-1}$.

Computing $V_d$, the volume of the $d$-dimensional simplex with vertices $(\sigma_i^{(d)})_{1 \leq i \leq d+1}$
can be done using the generalisation of Heron's formula, known as the Cayley-Menger determinant \cite{Sommer}. 
In view of the fact that the distance between any two vertices is fixed, namely equal to $\delta_d$, we obtain
$$
V_d = \left( \frac{1}{(d!)^2\, 2^d}\  \delta_d^{2d}\ | \det U_d | \right)^{1/2},
$$
where $U_d$ is the $d \times d$ matrix with ones everywhere except on the diagonal which is filled with twos.
One can check by induction that $\det U_d= d+1$ \cite{MarcelB}.
Finally we obtain:
$$
V_d = \left(  \frac{1}{(d!)^2\ 2^d}\   \left( 2\ \frac{d+1}{d}  \right) ^{d}  (d+1) \right)^{1/2}
=  \frac{1}{d!}  \left( \frac{(d+1)^{d+1}}{d^d} \right)^{1/2}.
$$

To obtain the volume which appears naturally here in our problem, that is the one of
$
L = L (\sigma_1^{(d)}, \sigma_2^{(d)},\dots,$ $\sigma_{d+1}^{(d)}),
$
we must multiply $V_d$ by a certain factor which takes into account the fact that we do not consider 
a simplex but a convex hull of the form $L (\sigma_1^{(d)}, \sigma_2^{(d)}, \dots, \sigma_{d+1}^{(d)})$. 
Indeed we may find this multiplicative factor from the fact that while a unit cube has volume 1, 
the standard simplex has a volume easily seen to be
$$
\int_{0 \leq x_i \leq 1,\ x_1+ \cdots + x_{d} \leq 1}  1 \quad {\rm d}x_1\cdots {\rm d} x_{d} =\frac{1}{d\ !} .
$$

This leads us to $d!\,V_d$ as the volume of the polyhedron $L(\sigma_1^{(d)}, \sigma_2^{(d)}, \dots, \sigma_{d+1}^{(d)})$ 
in dimension $d$. 

Conjecture \ref{conj78}, which generalises our results in dimensions 2 and 3, states that 
we have to multiply this by the homothetic factor $m^d$ to obtain the correct asymptotic value 
for $\DD^{(d)} (\Bcal^{(d+1)}_m)$, namely $d!\,V_d\ m^d$.

\end{document}